\documentclass[11pt]{amsart}
\usepackage{amssymb,amsmath,amsthm,enumitem,colonequals,mlmodern,tikz-cd,microtype}
\usepackage[cal=euler,bfcal,bb=px,bfbb]{mathalpha}
\usepackage{epsfig}
\usetikzlibrary{positioning}
\usepackage{url}
\usepackage[colorlinks=true,citecolor=blue,urlcolor=blue,linkcolor=blue]{hyperref}
\usepackage{setspace}
\usepackage{fancyhdr}
\usepackage{pdfpages}
\usepackage{color}
\usepackage{verbatim}
\usepackage{fancyhdr}
\usepackage[top=2.5cm, bottom=2.5cm, left=2cm, right=2cm]{geometry}
\usepackage{xcolor}
\colorlet{darkblue}{blue!55!black}
\colorlet{darkcyan}{cyan!50!black}
\colorlet{darkgreen}{green!60!black} 

\PassOptionsToPackage{hyphens}{url}
\usepackage{hyperref}
\hypersetup{
    colorlinks=true,
    linkcolor=darkblue,
    urlcolor=darkcyan,
    citecolor=darkgreen,
}

\def\eqref#1{\textcolor{darkblue}{(\ref{#1})}}

\usepackage[nameinlink]{cleveref} 
\Crefformat{section}{#2\S#1#3}
\Crefmultiformat{section}{#2\S\S#1#3}{ and~#2#1#3}{, #2#1#3}{, and~#2#1#3}

\usepackage[pagewise]{lineno}
\let\oldequation\equation
\let\oldendequation\endequation
\renewenvironment{equation}{\linenomathNonumbers\oldequation}{\oldendequation\endlinenomath}
\expandafter\let\expandafter\oldequationstar\csname equation*\endcsname
\expandafter\let\expandafter\oldendequationstar\csname endequation*\endcsname
\renewenvironment{equation*}{\linenomathNonumbers\oldequationstar}{\oldendequationstar\endlinenomath}
\let\oldalign\align
\let\oldendalign\endalign
\renewenvironment{align}{\linenomathNonumbers\oldalign}{\oldendalign\endlinenomath}
\expandafter\let\expandafter\oldalignstar\csname align*\endcsname
\expandafter\let\expandafter\oldendalignstar\csname endalign*\endcsname
\renewenvironment{align*}{\linenomathNonumbers\oldalignstar}{\oldendalignstar\endlinenomath}

\newcommand {\Hom} {\mathsf{Hom}}
\newcommand {\End} {\mathcal E\mathsf{nd}}

\newcommand  {\rank}     {\textsf{rank}}

\newcommand  {\id}   {\mathrm{id}}

\DeclareMathOperator{\Spec}{Spec}

\DeclareMathOperator{\tr}{tr}

\theoremstyle{plain}
\newtheorem{theorem}{Theorem}[section]
\newtheorem{lemma}[theorem]{Lemma}

\newtheorem{proposition}[theorem]{Proposition}

\theoremstyle{definition}
\newtheorem{definition}[theorem]{Definition}

\newtheorem{remark}[theorem]{Remark}

\AddToHook{env/conjecture/begin}{\crefalias{theorem}{conjecture}}
\AddToHook{env/lemma/begin}{\crefalias{theorem}{lemma}}
\AddToHook{env/corollary/begin}{\crefalias{theorem}{corollary}}
\AddToHook{env/proposition/begin}{\crefalias{theorem}{proposition}}
\AddToHook{env/definition/begin}{\crefalias{theorem}{definition}}
\AddToHook{env/remark/begin}{\crefalias{theorem}{remark}}
\AddToHook{env/example/begin}{\crefalias{theorem}{example}}
\AddToHook{env/hypothesis/begin}{\crefalias{theorem}{hypothesis}}
\AddToHook{env/notation/begin}{\crefalias{theorem}{notation}}

\numberwithin{equation}{section}
\numberwithin{theorem}{section}

\title{Weil's Theorem for Logarithmic Connections on Irreducible Nodal Curves}
\author{Sourav Das}
\email{sdas6565@gmail.com}
\begin{document}
\maketitle
\begin{abstract}
We establish an analogue of André Weil's classical theorem for irreducible nodal curves. Let \(X_0\) be an irreducible projective nodal curve. We prove that an indecomposable vector bundle or torsion-free coherent sheaf \(E\) on \(X_0\) admits a holomorphic logarithmic connection \(\nabla\colon E\to E\otimes\omega_{X_0}\) with respect to the dualizing sheaf if and only if \(\deg E=0\). Moreover, when \(\deg E=0\), such a connection can be chosen so that the induced logarithmic connection on the normalization has scalar residues \(\lambda\cdot I\) at one preimage of the node and \(-\lambda\cdot I\) at the other, for some \(\lambda\in\mathbb{C}\). Explicit one-parameter families of flat connections are constructed on the irreducible rational nodal cubic curve as an illustration.
\end{abstract}

\section{Introduction}

A classical theorem of Andr\'e Weil states that an indecomposable holomorphic vector bundle on a smooth projective curve admits a holomorphic logarithmic connection if and only if its degree is zero \cite{Weil}. Atiyah later reinterpreted this result cohomologically by introducing the \emph{Atiyah class}, which encodes the primary obstruction to the existence of such a connection \cite{Ati57}. Together with the Narasimhan--Seshadri theorem \cite{NS}, which establishes a bijection between stable vector bundles of degree zero and unitary representations of the fundamental group, Weil's theorem lies at the heart of non-abelian Hodge theory. This correspondence---relating holomorphic logarithmic connections, Higgs bundles, and representations of $\pi_1$---has profoundly shaped modern algebraic geometry and differential geometry.

In the singular setting, the picture is far more subtle. For \emph{irreducible nodal curves}, the appropriate objects are torsion-free coherent sheaves (equivalently, generalized parabolic bundles on the normalization together with gluing data at the nodes). Usha Bhosle developed a comprehensive theory of generalized parabolic bundles and proved a partial analogue of the Narasimhan--Seshadri correspondence: a degree-zero generalized parabolic bundle arises from a representation of the fundamental group of the nodal curve if and only if it is a direct sum of indecomposable degree-zero objects \cite{Bhosle92,Bhosle95}. However, unlike the smooth case, there exist stable torsion-free sheaves of degree zero on nodal curves that do \emph{not} correspond to any representation of $\pi_1(X_0)$. Thus a full non-abelian Hodge correspondence remains unsolved in the singular curve setting.

The goal of this note is to establish a \emph{direct Weil-type criterion on the nodal curve itself}, providing a natural bridge between the theory of holomorphic logarithmic connections and the geometry of singular curves. Let $X_0$ be an irreducible projective nodal curve of arithmetic genus $g\geq 1$, obtained by identifying two points $x^+$ and $x^-$ on its normalization $\pi\colon \widetilde{X_0}\to X_0$. Let $\omega_{X_0}$ denote the dualizing sheaf of $X_0$.

Our main results are the following analogues of Weil's theorem:

\begin{theorem}[Locally free case]
Let $E$ be an indecomposable locally free $\mathcal{O}_{X_0}$-module of rank $n$ on $X_0$. Then $E$ admits a holomorphic logarithmic connection $\nabla\colon E\to E\otimes\omega_{X_0}$ if and only if $\deg E=0$.
\end{theorem}

\begin{theorem}[Scalar residues -- locally free case]
Let $E$ be an indecomposable locally free $\mathcal{O}_{X_0}$-module of rank $n$ and degree zero. Then there exists a holomorphic logarithmic connection $\nabla$ on $E$ such that the induced logarithmic connection on the normalization $\widetilde{E}=\pi^*E$ has scalar residues $\lambda\cdot I$ at $x^+$ and $-\lambda\cdot I$ at $x^-$, for some $\lambda\in\mathbb{C}$.
\end{theorem}

We prove analogous statements for indecomposable torsion-free coherent sheaves:

\begin{theorem}[Torsion-free case]
Let $\mathcal{F}$ be an indecomposable torsion-free coherent sheaf of rank $n$ on $X_0$. Then $\mathcal{F}$ admits a holomorphic logarithmic connection $\nabla\colon \mathcal{F}\to \mathcal{F}\otimes\omega_{X_0}$ if and only if $\deg\mathcal{F}=0$.
\end{theorem}

\begin{theorem}[Scalar residues -- torsion-free case]
Let $\mathcal{F}$ be an indecomposable torsion-free coherent sheaf of rank $n$ and degree zero on $X_0$. Then there exists a holomorphic logarithmic connection on $\mathcal{F}$ inducing a logarithmic connection on the associated vector bundle $\widetilde{E}$ on the normalization with scalar residues $\lambda\cdot I$ at $x^+$ and $-\lambda\cdot I$ at $x^-$, for some $\lambda\in\mathbb{C}$.
\end{theorem}

These results furnish a clean existence criterion for holomorphic logarithmic connections directly on the singular curve, complementing Bhosle's representation-theoretic work. The proofs adapt Atiyah's classical strategy to the dualizing sheaf $\omega_{X_0}$, with careful analysis of gluing conditions and residue compatibility at the node. We also establish precise comparisons between Atiyah classes on $X_0$ and on its normalization. As a concrete illustration, we construct explicit one-parameter families of flat connections on indecomposable bundles on the rational nodal cubic curve.

\section{holomorphic logarithmic connections: Locally free case}

In this section we study holomorphic logarithmic connections on locally free sheaves on the nodal curve $X_0$. Let $X_0$ be an irreducible projective nodal curve of arithmetic genus $g\geq 1$, obtained by identifying two points $x^+$ and $x^-$ on its normalization $\pi\colon \widetilde{X_0}\to X_0$. Let $\omega_{X_0}$ denote the dualizing sheaf of $X_0$. We begin by recalling the definition of the dualizing sheaf $\omega_{X_0}$.

There are natural identifications 
\[
\Omega^1_{\widetilde{X_0}}(x^++x^-))_{x^+} \xrightarrow{\sim} \mathbb{C}, 
\qquad
\Omega^1_{\widetilde{X_0}}(x^++x^-))_{x^-} \xrightarrow{\sim} \mathbb{C}
\]
given by the Poincaré residue map (in a local coordinate $z$ near the point, it sends $dz/z \mapsto 1$).

The dualizing sheaf $\omega_{X_0}$ can be defined via the following exact sequence:
\begin{equation}
0 \to \omega_{X_0} \to \pi_*\Omega^1_{\widetilde{X_0}}(x^++x^-) \to \mathbb{C} \to 0,
\end{equation}
where the surjection $\pi_*\Omega^1_{\widetilde{X_0}}(x^++x^-)\to\mathbb{C}$ is the composition
\[
\pi_*\Omega^1_{\widetilde{X_0}}(x^++x^-)\ \longrightarrow\ 
\bigl(\Omega^1_{\widetilde{X_0}}(x^++x^-)\bigr)_{x^+} \oplus 
\bigl(\Omega^1_{\widetilde{X_0}}(x^++x^-)\bigr)_{x^-}
\ \longrightarrow\ \mathbb{C},
\]
and the second arrow is the sum of the two Poincaré residue maps.

\begin{definition}
    A \emph{holomorphic logarithmic connection} on a locally free $\mathcal{O}_{X_0}$-module $E$ is a $\mathbb{C}$-linear morphism 
    \[
    \nabla \colon E \to E \otimes \omega_{X_0}
    \]
    satisfying the Leibniz rule.
\end{definition}

\begin{remark}
    Let $(E, \nabla)$ be a holomorphic logarithmic connection on $X_0$. On the normalization $\widetilde{X_0}$, the pullback $\widetilde{E} := \pi^*E$ carries a logarithmic connection
    \[
    \widetilde{\nabla} := \pi^*\nabla \colon \widetilde{E} \to \widetilde{E} \otimes \Omega^1_{\widetilde{X_0}}(x^++x^-).
    \]
    Moreover, the gluing isomorphism $A \colon \widetilde{E}_{x^+} \xrightarrow{\sim} \widetilde{E}_{x^-}$ of the vector bundle $E$ satisfies the residue compatibility condition
    \[
    -A \circ \operatorname{res}_{x^-}\widetilde{\nabla} = \operatorname{res}_{x^+}\widetilde{\nabla} \circ A.
    \]
    Equivalently, up to the identification given by $A$, we have $\operatorname{res}_{x^-}\widetilde{\nabla} = -\operatorname{res}_{x^+}\widetilde{\nabla}$. (The minus sign arises because the two fibers of $\omega_{X_0}$ at the node are identified by multiplication by $-1$ in the definition of the dualizing sheaf.)
\end{remark}

\begin{remark}
    The residue of a connection $(E,\nabla)$ at the node $x$ is well-defined only up to scalar. More precisely, the evaluation of $\nabla$ at $x$ gives a map
    \[
    \nabla_x \colon E_x \to E_x \otimes \omega_{X_0,x},
    \]
    which defines an element of $\End(E_x) \otimes \omega_{X_0,x}$. Since $\omega_{X_0,x} \cong \mathbb{C}$, this corresponds to an endomorphism of $E_x$ only after choosing a local generator of $\omega_{X_0,x}$.
\end{remark}

\section{Analogue of Andre Weil's theorem: Locally free case}

Let $E$ be a locally free $\mathcal{O}_{X_0}$-module. The obstruction to the existence of a holomorphic logarithmic connection $\nabla: E\to E\otimes \omega_{X_0}$ is the Atiyah class
\[
[At(E)]\in H^1(X_0, \End E\otimes \omega_{X_0}),
\]
defined as follows.

\begin{definition}
    Let $E$ be a locally free $\mathcal{O}_{X_0}$-module. Choose an affine open covering $\{U_i\}$ of $X_0$ such that $E|_{U_i}\xrightarrow{\phi_i} \mathcal{O}_{U_i}^{\oplus n}$ is an isomorphism for each $i$. For each $i$, choose any connection $\nabla_i: E|_{U_i}\to E|_{U_i}\otimes \omega_{U_i}$ using the trivialization $\phi_i$. Set
    \[
    \alpha_{ij} := \nabla_j - \nabla_i \in \Gamma(U_{ij}, \End E \otimes \omega_{X_0}).
    \]
    The collection $\{\alpha_{ij}\}$ is a Čech $1$-cocycle with values in $\End E \otimes \omega_{X_0}$, and its cohomology class
    \[
    At(E) = [\{\alpha_{ij}\}] \in H^1(X_0, \End E \otimes \omega_{X_0})
    \]
    is called the \emph{Atiyah class} of $E$.
\end{definition}
    \smallskip
    \textbf{Description in terms of Čech cocycles.}
Suppose \(E|_{U_i} \cong \mathcal{O}_{U_i}^{\oplus n}\) via trivializations \(\phi_i\) for \(i \in I\). On each \(U_i\), choose a local connection of the form
\[
\nabla_i = d + A_i,
\]
where \(A_i\) is a matrix-valued section of \(\omega_{X_0}|_{U_i}\).

These local connections glue to a global holomorphic logarithmic connection on \(E\) if and only if they are compatible on overlaps \(U_{ij} = U_i \cap U_j\). Let \(\phi_{ij} = \phi_i \circ \phi_j^{-1}\) be the transition function. A direct computation shows that the compatibility condition
\[
\phi_i^{-1} \circ (\nabla_i) \circ \phi_i = \phi_j^{-1} \circ (\nabla_j) \circ \phi_j
\]
is equivalent to
\[
A_j - \phi_{ij}^{-1} A_i \phi_{ij} = \phi_{ij}^{-1} \, d(\phi_{ij}).
\]
Thus the differences \(\phi_{ij}^{-1} \, d(\phi_{ij})\) form a Čech \(1\)-cocycle with values in \(\End(E) \otimes \omega_{X_0}\). The cohomology class of this cocycle is independent of the choices of trivializations and local connections, and is called the \emph{Atiyah class} of \(E\):
\[
At(E) = \bigl[ \phi_{ij}^{-1} \, d(\phi_{ij}) \bigr] \in H^1(X_0, \End(E) \otimes \omega_{X_0}).
\]

\begin{lemma}
    On a projective nodal curve \(X_0\), the two definitions of the first Chern class / degree of a line bundle \(L\) coincide:
    \[
    c_1(L) \;\in\; H^1(X_0, \omega_{X_0}) \qquad \text{(Čech definition)}
    \]
    corresponds, under the canonical isomorphism \(H^1(X_0, \omega_{X_0}) \cong \mathbb{C}\), to the integer
    \[
    \deg(L) := \chi(L) - \chi(\mathcal{O}_{X_0}) \;\in\; \mathbb{Z} \qquad \text{(numerical definition)}.
    \]
\end{lemma}

\begin{proof}
    Since \(X_0\) is Gorenstein (as any nodal curve is), the dualizing sheaf \(\omega_{X_0}\) is invertible, and the Riemann--Roch theorem holds for every line bundle on \(X_0\):
    \[
    \chi(L) = \deg(L) + 1 - p_a(X_0),
    \]
    where \(p_a(X_0)\) is the arithmetic genus and \(\deg(L)\) is defined numerically as \(\chi(L) - \chi(\mathcal{O}_{X_0})\). Thus it suffices to show that the Čech class \(c_1(L) \in H^1(X_0, \omega_{X_0})\) corresponds exactly to this numerical degree under the trace/residue isomorphism \(H^1(X_0, \omega_{X_0}) \cong \mathbb{C}\).

    Let \(\{U_i\}\) be an open affine cover of \(X_0\) on which \(L\) is trivial, and let \(g_{ij} \in \mathcal{O}^*(U_i \cap U_j)\) be the transition functions of \(L\). The first Chern class is defined by
    \[
    c_1(L) = \bigl[ \{ d\log g_{ij} \} \bigr] \in H^1(X_0, \omega_{X_0}),
    \]
    where \(d\log g_{ij} = \frac{d g_{ij}}{g_{ij}}\).

    Now choose a rational (meromorphic) section \(s\) of \(L\). On each \(U_i\), we can write \(s = f_i \cdot e_i\), where \(e_i\) is a local trivializing section and \(f_i \in K(X_0)^*\) (the function field). Then on overlaps we have
    \[
    f_i = g_{ij} f_j,
    \]
    so
    \[
    d\log g_{ij} = d\log(f_i / f_j) = \frac{df_i}{f_i} - \frac{df_j}{f_j}.
    \]
    Thus the cocycle \(\{d\log g_{ij}\}\) is precisely the Čech representative of the divisor \(\operatorname{div}(s)\) (the divisor of zeros minus poles of \(s\)) with values in \(\omega_{X_0}\) (the counts of zeros and poles are converted into computation of residues of the differential form).

    By the residue theorem on the Gorenstein curve \(X_0\) (or equivalently, by the definition of the trace map \(\operatorname{Tr} : H^1(X_0, \omega_{X_0}) \to \mathbb{C}\)), the image of the class \([\{d\log g_{ij}\}]\) in \(\mathbb{C}\) is exactly the sum of the residues of the logarithmic differential, which equals the degree of the divisor \(\operatorname{div}(s)\):
    \[
    \operatorname{Tr}\bigl(c_1(L)\bigr) = \deg(\operatorname{div}(s)).
    \]

    Since this is independent of the choice of the rational section \(s\), and the numerical degree \(\deg(L)\) is defined to be the degree of any divisor associated to \(L\) (i.e., \(\deg(\operatorname{div}(s))\)), we obtain
    \[
    \operatorname{Tr}\bigl(c_1(L)\bigr) = \deg(L).
    \]
    Therefore the two definitions coincide.

    This completes the proof.
\end{proof}

\begin{lemma}
    The trace of the Atiyah class coincides with the first Chern class:
    \[
    \operatorname{tr}\bigl(\operatorname{At}(E)\bigr) = c_1(E) \in H^1(X_0, \omega_{X_0}) \cong \mathbb{C}.
    \]
\end{lemma}

\begin{proof}
    Assume \(E\) is a vector bundle on the nodal curve \(X_0\) (i.e., locally free). Let \(\{U_i\}\) be an open cover of \(X_0\) such that \(E\) is trivial on each \(U_i\), with transition functions \(\phi_{ij} \in \operatorname{GL}(r, \mathcal{O}_{U_i \cap U_j})\).

    The Atiyah class \(\operatorname{At}(E)\) is represented (in Čech cohomology) by the cocycle
    \[
    a_{ij} = \phi_{ij}^{-1} \, d\phi_{ij} \in H^1(X_0, \End(E) \otimes \omega_{X_0}).
    \]

    Applying the trace morphism \(\operatorname{tr} \colon \End(E) \to \mathcal{O}_{X_0}\) gives
    \[
    \operatorname{tr}(a_{ij}) = \operatorname{tr}(\phi_{ij}^{-1} \, d\phi_{ij}) = d\log(\det \phi_{ij}).
    \]

    The 1-cocycle \(\{d\log(\det \phi_{ij})\}\) is the standard Čech representative of the first Chern class \(c_1(\det E)\) in \(H^1(X_0, \omega_{X_0})\). Since \(c_1(E) := c_1(\det E)\), we obtain
    \[
    \operatorname{tr}\bigl(\operatorname{At}(E)\bigr) = c_1(E) \in H^1(X_0, \omega_{X_0}).
    \]
\end{proof}

\begin{remark}
    Consequently, the vanishing of the Atiyah class $At(E)$ implies $c_1(E) = 0$. In particular, $\deg E = 0$. Indeed, since $\deg E = \deg(\pi^*E)$, it suffices to show that $c_1(\pi^*E) = 0$. One may choose the affine open cover $\{U_i\}$ so that all overlaps $U_{ij}$ avoid the node. On such a cover the transition functions of $E$ coincide with those of its pullback $\pi^*E$, and thus $[d\log(\det \phi_{ij})]$ is precisely the first Chern class of $\pi^*E$. Hence $c_1(E) = 0$ forces $\deg E = 0$.
\end{remark}

\begin{remark}
    The relationship between the Atiyah classes on the nodal curve and its normalization can be made precise as follows. Consider the short exact sequence of sheaves
    \[
    0 \to \pi_*\Omega^1_{\widetilde{X_0}} \to \omega_{X_0} \to \mathcal{O}_x \to 0.
    \]
    The associated long exact sequence in cohomology begins with
    \[
    0 \to H^0(\widetilde{X_0}, \Omega^1_{\widetilde{X_0}}) \to H^0(X_0, \omega_{X_0}) \to \mathbb{C} \to H^1(\widetilde{X_0}, \Omega^1_{\widetilde{X_0}}) \to H^1(X_0, \omega_{X_0}) \to 0.
    \]
    Since $\dim H^0(\widetilde{X_0}, \Omega^1_{\widetilde{X_0}}) = g-1$ and $\dim H^0(X_0, \omega_{X_0}) = g$ (where $g$ denotes the arithmetic genus of $X_0$), the connecting homomorphism $\mathbb{C} \to H^1(\widetilde{X_0}, \Omega^1_{\widetilde{X_0}})$ is surjective, and the induced map
    \[
    H^1(\widetilde{X_0}, \Omega^1_{\widetilde{X_0}}) \xrightarrow{\sim} H^1(X_0, \omega_{X_0})
    \]
    is an isomorphism. It follows that the Atiyah class $At(E) = [\phi_{ij}^{-1}d\phi_{ij}] \in H^1(X_0, \omega_{X_0})$ is the image of the Atiyah class $At(\widetilde{E}) \in H^1(\widetilde{X_0}, \End\widetilde{E} \otimes \Omega^1_{\widetilde{X_0}})$ of the pullback bundle under this isomorphism. In other words, the two classes contain equivalent information.
\end{remark}

\begin{theorem}\label{Weil}
    Let $E$ be an indecomposable vector bundle of rank $n$ on the nodal curve $X_0$. Then $E$ admits a holomorphic logarithmic connection $\nabla: E\to E\otimes \omega_{X_0}$ if and only if $\deg E = 0$.
\end{theorem}
\begin{proof}
    It suffices to show that $[At(E)] = 0$ if and only if $\deg E = 0$. Consider the perfect pairing coming from Serre duality:
    \[
    H^0(\End E) \times H^1(X_0, \End E \otimes \omega_{X_0}) \xrightarrow{\cup} H^1(X_0, \End E \otimes \omega_{X_0}) \xrightarrow{\operatorname{tr}} H^1(X_0, \omega_{X_0}) \cong \mathbb{C}.
    \]
    Thus we need $\langle \phi, At(E) \rangle = 0$ for all $\phi \in \End E$ if and only if $\deg E = 0$.

    Since $E$ is indecomposable, every global endomorphism has the form $\phi = \lambda \cdot \operatorname{Id}_E + N$, where $\lambda \in \mathbb{C}$ and $N$ is nilpotent. Since $E$ is indecomposable, every global endomorphism has the form $\phi = \lambda \cdot \operatorname{Id}_E + N$, where $\lambda \in \mathbb{C}$ and $N$ is nilpotent. This is a consequence of the well-known Fitting Lemma \cite[Proposition 7.4, Page 441]{Lang} and the fact that $End(E)$ is a finite dimensional vector space. 
    
    We first show that $\langle N, At(E) \rangle = 0$ for any nilpotent endomorphism $N$. 

Since \(N \in H^0(X_0, \End(E))\) is nilpotent, i.e., \(N^k = 0\) for some \(k \geq 1\), one can construct a complete flag of subbundles
\[
0 = E_0 \subset E_1 \subset E_2 \subset \cdots \subset E_r = E,
\]
where each \(E_i\) is a subbundle of rank \(i\), such that \(N(E_i) \subset E_{i-1}\) for all \(i\).

This flag (often called the \emph{nilpotent flag} associated to \(N\)) can be constructed explicitly as follows. Begin with the kernel filtration (also known as the Fitting filtration) of \(N\):
\[
0 \subset \ker N \subset \ker(N^2) \subset \ker(N^3) \subset \cdots \subset \ker(N^k) = E.
\]
Each \(\ker(N^j)\) is a subbundle of \(E\) because \(N\) is a morphism of vector bundles. This filtration may have jumps of rank greater than 1. To obtain a complete flag (i.e., \(\rank(E_i/E_{i-1}) = 1\) for each \(i\)), refine it by successively choosing complementary subbundles.

More precisely, suppose we have already constructed \(E_{i-1}\). Since \(N\) is nilpotent, the restriction \(N \colon E / E_{i-1} \to E / E_{i-2}\) is well-defined. One can choose a rank-one subbundle \(L_i \subset E / E_{i-1}\) such that \(N(L_i) \subset E_{i-1}/E_{i-2}\). Let \(E_i\) be the preimage of \(L_i\) under the quotient map \(E \twoheadrightarrow E / E_{i-1}\). Repeating this process yields the desired complete flag satisfying \(N(E_i) \subset E_{i-1}\) for all \(i\). 

One can choose a trivializing affine open cover \(\{U_i\}\) of \(X_0\) together with local frames \(\phi_i\) adapted to the flag such that, with respect to these frames, the transition functions \(g_{ij} = \phi_i \circ \phi_j^{-1}\) are upper triangular and the matrix representing \(N\) is strictly upper triangular.

It follows that the Čech cocycle representative \(\phi_{ij}^{-1} d(\phi_{ij})\) of the Atiyah class \(\operatorname{At}(E)\) is also upper triangular. Consequently, on each overlap, the pointwise matrix product
\[
N \cdot \bigl( \phi_{ij}^{-1} d(\phi_{ij}) \bigr)
\]
is strictly upper triangular. The trace of a strictly upper triangular matrix vanishes, so
\[
\operatorname{tr}\bigl( N \cdot (\phi_{ij}^{-1} d(\phi_{ij})) \bigr) = 0.
\]
This shows that the pair \(\langle N, \operatorname{At}(E) \rangle = 0\).

On the other hand, for the scalar endomorphism, we have
\[
\langle \lambda \cdot \operatorname{Id}_E, \operatorname{At}(E) \rangle 
= n\lambda \cdot \langle \operatorname{Id}_E, \operatorname{At}(E) \rangle 
= n\lambda \cdot \operatorname{tr}\bigl(\operatorname{At}(E)\bigr) 
= n\lambda \cdot c_1(E).
\]
Therefore, \(\langle \phi, \operatorname{At}(E) \rangle = 0\) for every global endomorphism \(\phi \in H^0(X_0, \End E)\) if and only if \(c_1(E) = 0\), which is equivalent to \(\deg E = 0\).
\end{proof}

\begin{theorem}\label{Weil2}
    Let $E$ be an indecomposable vector bundle of rank $n$ and degree $0$. Then $E$ admits a holomorphic logarithmic connection $\nabla: E\to E\otimes \omega_{X_0}$ such that the induced logarithmic connection $(\widetilde{E} := \pi^*E, \widetilde{\nabla}: \widetilde{E}\to \widetilde{E}\otimes \Omega^1_{\widetilde{X_0}}(x^++x^-))$ has residues of the form $\lambda \cdot I_{\widetilde{E}_{x^+}}$ at $x^+$ and $-\lambda \cdot I_{\widetilde{E}_{x^-}}$ at $x^-$, for some scalar $\lambda \in \mathbb{C}$.
\end{theorem}
\begin{proof}
Let $\mathcal{F}$ denote the subsheaf of $\pi_*(\End \widetilde{E})$ defined as follows. Consider the diagram
\begin{equation}
    \begin{tikzcd}[column sep=1.8em, row sep=1.8em]
    0 \arrow{r} 
      & \pi_*(\End\widetilde{E}\otimes I_{\{x^+,x^-\}}) \arrow["="]{d} \arrow{r} 
      & j^{-1}(\End\widetilde{E}_{x^+}) = \End E \arrow[hook]{d} \arrow{r} 
      & \pi_*(\End\widetilde{E}_{x^+}) \arrow[hook,"\Delta"]{d} \arrow{r} 
      & 0 \\
    0 \arrow{r} 
      & \pi_*(\End\widetilde{E}\otimes I_{\{x^+,x^-\}}) \arrow{r} 
      & \pi_*(\End\widetilde{E}) \arrow["j"]{r} 
      & \pi_*(\End\widetilde{E}_{x^+}\oplus \End\widetilde{E}_{x^-}) \arrow{r} 
      & 0 \\
    0 \arrow{r} 
      & \pi_*(\End\widetilde{E}\otimes I_{\{x^+,x^-\}}) \arrow["="]{u} \arrow{r} 
      & j^{-1}(\End\widetilde{E}_{x^+}) \arrow[hook]{u} \arrow{r} 
      & \pi_*(\End\widetilde{E}_{x^+}) \arrow[hook,"\Delta^{Op}"]{u} \arrow{r} 
      & 0 \\
    0 \arrow{r} 
      & \pi_*(\End\widetilde{E}\otimes I_{\{x^+,x^-\}}) \arrow["="]{u} \arrow{r} 
      & \mathcal{F} := j^{-1}(\mathbb{C}I_{\widetilde{E}_{x^+}}) \arrow[hook]{u} \arrow{r} 
      & \mathbb{C}I_{\widetilde{E}_{x^+}} \arrow[hook]{u} \arrow{r} 
      & 0
    \end{tikzcd}
\end{equation}
Here $\Delta: \pi_*(\End\widetilde{E}_{x^+})\hookrightarrow \pi_*(\End\widetilde{E}_{x^+}\oplus \End\widetilde{E}_{x^-})$ sends $\phi\mapsto (\phi, A\circ \phi\circ A^{-1})$, and $\Delta^{Op}: \pi_*(\End\widetilde{E}_{x^+})\hookrightarrow \pi_*(\End\widetilde{E}_{x^+}\oplus \End\widetilde{E}_{x^-})$ sends $\phi\mapsto (\phi, -A\circ \phi\circ A^{-1})$, where $A: \widetilde{E}_{x^+}\to \widetilde{E}_{x^-}$ is the isomorphism that determines the vector bundle $E$ on the nodal curve $X_0$. By an abuse of notation, we denote by $\Delta$ and $\Delta^{Op}$ the respective images of these maps.

It is straightforward to check that $E$ admits a connection with the stated residue property if and only if the obstruction class $\beta_E \in H^1(X_0, \omega_{X_0}\otimes \mathcal{F})$ vanishes.

By Serre duality we have a non-degenerate pairing
\[
\langle -,-\rangle : H^0(X_0, \mathcal{F}^\vee) \times H^1(X_0, \omega_{X_0}\otimes \mathcal{F}) \to H^1(X_0, \omega_{X_0}) \cong \mathbb{C}.
\]

Now we need an interpretation of the vector space $H^0(X_0, \mathcal{F}^\vee)$. First, note that $\pi_*(\End\widetilde{E}\otimes I_{\{x^+,x^-\}}) \cong \End E \otimes I_x$, where $I_x$ is the ideal sheaf of the node. From the diagram we obtain the short exact sequence
\[
0 \to \End E \otimes I_x \to \mathcal{F} \to \Delta^{Op}(\mathbb{C} \cdot I_{\widetilde{E}_{x^+}}) \to 0.
\]
Dualizing yields
\begin{equation}\label{diag1}
\begin{tikzcd}
0\arrow{r} & \mathcal{F}^\vee\arrow{r} & (\End E)^\vee \otimes I_x^\vee \arrow{r}\arrow["\cong"]{d} & \mathcal Ext^1(\Delta^{Op}(\mathbb{C}\cdot I_{\widetilde{E}_{x^+}}), \mathcal{O}_{X_0})\arrow{r}\arrow["\cong"]{d} & \cdots \\
& & \End E \otimes \pi_*\mathcal{O}_{\widetilde{X_0}} & \Delta^{Op}(\mathbb{C}\cdot I_{\widetilde{E}_{x^+}})
\end{tikzcd}
\end{equation}
(Note that $I_x^\vee \cong \pi_*\mathcal{O}_{\widetilde{X_0}}$ canonically, and $(\End E)^\vee \cong \End E$ via the trace pairing.) Therefore,
\[
H^0(X_0, \mathcal{F}^\vee) = Ker \Bigl[ H^0(\End E \otimes \pi_*\mathcal{O}_{\widetilde{X_0}}) = H^0(\widetilde{X_0}, \pi^*\End E) \to \Delta^{Op}(\mathbb{C}\cdot I_{\widetilde{E}_{x^+}}) \Bigr].
\]
Analysing this map, we obtain
\begin{equation}\label{eqnew}
H^0(X_0, \mathcal{F}^\vee) = \{\phi \in \End E \mid \operatorname{tr}((\pi^*\phi)_{x^+}) = -\operatorname{tr}(\pi^*\phi)_{x^-}\}.
\end{equation}

Since $E$ is indecomposable, any global endomorphism is of the form $\phi = \lambda \cdot I + N$, where $\lambda$ is a scalar and $N$ is nilpotent. If $\phi$ also lies in $H^0(X_0, \mathcal{F}^\vee)$, then \eqref{eqnew} forces $\lambda = 0$. Thus $\phi$ is nilpotent. Using the same argument as in the proof of Proposition \ref{Weil}, we have $\langle N, \beta_E \rangle = 0$. Therefore $\beta_E = 0$, and the proof is complete.
\end{proof}

\section{holomorphic logarithmic connections on Torsion-Free Sheaves}

\begin{definition}
    Let \(\mathcal{F}\) be a torsion-free coherent sheaf on the nodal curve \(X_0\). A \emph{holomorphic logarithmic connection} on \(\mathcal{F}\) is a \(\mathbb{C}\)-linear morphism
    \[
    \nabla \colon \mathcal{F} \to \mathcal{F} \otimes \omega_{X_0}
    \]
    satisfying the Leibniz rule
    \[
    \nabla(f \cdot s) = f \cdot \nabla(s) + s \otimes df
    \]
    for all local sections \(f \in \mathcal{O}_{X_0}\) and \(s \in \mathcal{F}\).
\end{definition}

\subsection{Lift to the Normalization}

\begin{proposition}\label{Trip}
Let \(\mathcal{F}\) be a torsion-free coherent sheaf on the nodal curve \(X_0\) equipped with a holomorphic logarithmic connection \(\nabla \colon \mathcal{F} \to \mathcal{F} \otimes \omega_{X_0}\). Then there exists a \emph{unique} triple \((\widetilde{E}, A, \widetilde{\nabla})\) up to unique isomorphism, consisting of
\begin{itemize}
    \item a vector bundle \(\widetilde{E}\) on the normalization \(\pi \colon \widetilde{X_0} \to X_0\) of the same rank and degree as \(\mathcal{F}\),
    \item a linear gluing map \(A \colon \widetilde{E}_{x^+} \to \widetilde{E}_{x^-}\),
    \item a logarithmic connection \(\widetilde{\nabla} \colon \widetilde{E} \to \widetilde{E} \otimes \Omega^1_{\widetilde{X_0}}(x^+ + x^-)\),
\end{itemize}
such that the following hold:

\begin{enumerate}
    \item \(\mathcal{F}\) sits in the short exact sequence
    \[
    0 \to \mathcal{F} \to \pi_* \widetilde{E} \to \pi_* \Bigl( \frac{\widetilde{E}_{x^+} \oplus \widetilde{E}_{x^-}}{\Gamma_A} \Bigr) \to 0,
    \]
    where \(\Gamma_A\) is the graph of \(A\).

    \item The holomorphic logarithmic connection \(\nabla\) is compatible with \(\widetilde{\nabla}\) via the commutative diagram
    \[
    \begin{tikzcd}
    0 \arrow{r} & \mathcal{F} \arrow{r} \arrow["\nabla"]{d}
      & \pi_* \widetilde{E} \arrow{r} \arrow["\pi_* \widetilde{\nabla}"]{d}
      & \pi_* \Bigl( \frac{\widetilde{E}_{x^+} \oplus \widetilde{E}_{x^-}}{\Gamma_A} \Bigr) \arrow{r} \arrow["\pi_* \widetilde{\nabla}"]{d}
      & 0 \\
    0 \arrow{r} & \mathcal{F} \otimes \omega_{X_0} \arrow{r}
      & \pi_* \widetilde{E} \otimes \omega_{X_0} \arrow{r}
      & \pi_* \Bigl( \frac{\widetilde{E}_{x^+} \oplus \widetilde{E}_{x^-}}{\Gamma_A} \Bigr) \otimes \omega_{X_0} \arrow{r}
      & 0
    \end{tikzcd}
    \]

    This commutative diagram also explains how a compatible triple $(\tilde{E}, A, \widetilde{\nabla})$ induces a torsion-free sheaf $\mathcal F$ with a logarithmic connection $\nabla$.

    \item The residues of \(\widetilde{\nabla}\) satisfy the compatibility condition
    \[
    -A \circ \operatorname{res}_{x^-} \widetilde{\nabla} = \operatorname{res}_{x^+} \widetilde{\nabla} \circ A.
    \]

    \item For each \(\lambda, \mu \in \mathbb{C}\), let \(E^+_\lambda\) (resp.\ \(E^-_\mu\)) denote the generalized \(\lambda\)-eigenspace of \(\alpha\) on \(\widetilde{E}_{x^+}\) (resp.\ the generalized \(\mu\)-eigenspace of \(\beta\) on \(\widetilde{E}_{x^-}\)). Then
\[
A(E^+_\lambda) \subseteq E^-_{-\lambda}.
\]
\end{enumerate}

\end{proposition}
\begin{proof}
We first recall the construction of the pair \((\widetilde{E}, A)\) following \cite[pp.\ 104--106]{NSI}. The construction is local and extends to any nodal curve. We describe the technical details of the construction in the following steps. 

\begin{enumerate}
    \item Define 
    \[
    E \;:=\; \frac{\pi^*\mathcal{F}}{tors(\pi^*\mathcal{F})}.
    \]
    This is a vector bundle on \(\widetilde{X_0}\). The adjunction morphism induces a natural inclusion \(\mathcal{F} \hookrightarrow \pi_* E\), yielding the short exact sequence
    \[
    0 \to \mathcal{F} \longrightarrow \pi_* E \twoheadrightarrow T \to 0,
    \]
    where \(T := \pi_* E / \mathcal{F}\).

    \item Let \(N := \ker(E_{x^+} \oplus E_{x^-} \twoheadrightarrow T)\). By \cite{NSI}, \(N\) is the image of the canonical map \(\mathcal{F}_p \to (\pi_* E)_p\), and the projections \(N \to E_{x^+}\) and \(N \to E_{x^-}\) are surjective \cite[Remark 2.1]{NSI}.

    \item Set \(K := \ker(N \to E_{x^+})\). Then \(K \subseteq E_{x^-}\), because $N$ is the image of the canonical map \(\mathcal{F}_p \to (\pi_* E)_p\).

    \item Perform a Hecke modification \(i \colon E \twoheadrightarrow \widetilde{E}\) at \(x^-\) such that 
    \[
    \ker(i_{x^-} \colon E_{x^-} \to \widetilde{E}_{x^-}) = K.
    \]
    Let \(N^1\) be the image of \(N\) under the natural map 
    \[
    \theta \colon E_{x^+} \oplus E_{x^-} \to \widetilde{E}_{x^+} \oplus \widetilde{E}_{x^-}.
    \]
    Since no modification is performed at \(x^+\), the projection \(N^1 \twoheadrightarrow \widetilde{E}_{x^+}\) is an isomorphism (as \(\dim N^1 = rank \widetilde{E}\)). Moreover, \(N = \theta^{-1}(N^1)\).

    \item We have inclusions \(\mathcal{F} \hookrightarrow \pi_* E \hookrightarrow \pi_* \widetilde{E}\). Hence
    \[
    \mathcal{F} = \{ f \in \pi_* \widetilde{E} \mid \text{evaluation of } f \text{ at the node lies in } N^1 \}.
    \]
    The isomorphism \(N^1 \cong \widetilde{E}_{x^+}\) defines a unique linear map \(A \colon \widetilde{E}_{x^+} \to \widetilde{E}_{x^-}\) satisfying
    \[
    \mathcal{F} = \{ f \in \pi_* \widetilde{E} \mid A(f(x^+)) = f(x^-) \}.
    \]
\end{enumerate}

The uniqueness of the pair \((\widetilde{E}, A)\) is obvious from the construction. 

Now suppose \(\nabla\) is a holomorphic logarithmic connection on \(\mathcal{F}\). We construct the induced logarithmic connection \(\widetilde{\nabla}\) on \(\widetilde{E}\).

Since \(\pi\) is birational, there is a natural pullback connection \(\pi^*\nabla\) on \(\pi^*\mathcal{F}\). Because \(\pi^*\mathcal{F} = E \oplus tors(\pi^*\mathcal{F})\), it induces a logarithmic connection on \(E\), still denoted \(\pi^*\nabla \colon E \to E \otimes \Omega^1_{\widetilde{X_0}}(x^+ + x^-)\).

\

\textbf{Invariance of \(N\) and \(K\):}  
We have a commutative diagram
\[
\begin{tikzcd}
0 \rar & \mathcal{F} \rar \arrow[d,"\nabla"] & \pi_* E \rar \arrow[d,"\pi_*(\pi^*\nabla)"] & T \rar \arrow[dotted, "\nabla"]{d} & 0 \\
0 \rar & \mathcal{F} \otimes \omega_{X_0} \rar & \pi_* E \otimes \omega_{X_0} \rar & T \otimes \omega_{X_0} \rar & 0
\end{tikzcd}
\]
First, notice that from the construction of \(\pi^*\nabla\) on $E$ it follows that the square on the left commutes, and therefore induces a connection (also denoted by $\nabla$) on $T$. A straightforward diagram chase shows that \(\pi^*\nabla\) preserves \(N \subset E_{x^+} \oplus E_{x^-}\). Consequently, its residue at \(x^-\) preserves \(K = \ker(N \twoheadrightarrow E_{x^+})\).

\

\textbf{Connection preserves the Hecke modification:}  
 Let \((\pi^*\nabla)^*\) be the dual connection on \(E^*\) (which exists by the standard construction: \(\langle \nabla^* \phi, s \rangle + \langle \phi, \nabla s \rangle = d\langle \phi, s \rangle\)). We realize the Hecke modification using the dual bundle in the following steps.

\begin{enumerate}
    \item Let \(K^\perp \subset (E^*)_{x^-}\) be the orthogonal complement of \(K\) with respect to the canonical pairing.
    \item Define \(E' := \ker\bigl( E^* \twoheadrightarrow (E^*_{x^-})/K^\perp \bigr)\).
    \item Set \(\widetilde{E} := (E')^*\). This is the desired Hecke modification with \(\ker(\widetilde{E}_{x^-} \leftarrow E_{x^-}) = K\).
\end{enumerate}

The subspace \(K^\perp\) is invariant under \((\pi^*\nabla)^*\) because \(K\) is invariant under the residue of \(\pi^*\nabla\) and the pairing is compatible with the connections. The subspace \(K^\perp \subset (E^*)_{x^-}\) is invariant under the residue of the dual connection \((\pi^*\nabla)^*\). 

Fix \(p=x^-\). Let \(V=E_p\), \(V^*=(E^*)_p\), and let \(\langle\,\cdot\,,\,\cdot\,\rangle\) denote the canonical pairing. Let \(\alpha=\operatorname{res}_p(\widetilde{\nabla})\). By the diagram chase, \(K\subset V\) is \(\alpha\)-invariant. Define
\[
K^\perp=\{\phi\in V^*\mid \langle k,\phi\rangle=0\ \forall\,k\in K\}.
\]

The dual connection satisfies \(\operatorname{res}_p(\nabla^*)=-\alpha^*\), where \(\alpha^*\) is defined by the following property:
\[
\langle\alpha^*\phi,s\rangle:=\langle\phi,\alpha s\rangle.
\]

Now let \(\phi\in K^\perp\) and \(k\in K\). Then
\[
\langle k,\operatorname{res}_p(\nabla^*)\phi\rangle
= \langle k,(-\alpha^*)\phi\rangle
= -\langle\alpha k,\phi\rangle=0,
\]
Hence \(\operatorname{res}_p(\nabla^*)\phi\in K^\perp\).

Therefore, the Hecke modification \(E'=\ker\bigl(E^*\twoheadrightarrow(E^*_{x^-})/K^\perp\bigr)\) inherits a logarithmic connection. Dualizing then yields the desired connection \(\widetilde{\nabla}\) on \(\widetilde{E}\).

Finally, the residue compatibility
\[
-A \circ \operatorname{res}_{x^-} \widetilde{\nabla} = \operatorname{res}_{x^+} \widetilde{\nabla} \circ A
\]
follows from the commutative square
\[
\begin{tikzcd}
\widetilde{E}_{x^+} \arrow[r, "\operatorname{res}_{x^+} \widetilde{\nabla}"] \arrow[d, "A"] 
  & \widetilde{E}_{x^+} \arrow[d, "A"] \\
\widetilde{E}_{x^-} \arrow[r, "\operatorname{res}_{x^-} \widetilde{\nabla}"] 
  & \widetilde{E}_{x^-}
\end{tikzcd}
\]
which is obtained by evaluating the global commutative diagram
\[
\begin{tikzcd}
\mathcal{F} \arrow[r, "\nabla"] \arrow[d] 
  & \mathcal{F} \otimes \omega_{X_0} \arrow[d] \\
\pi_* \widetilde{E} \arrow[r, "\pi_* \widetilde{\nabla}"] 
  & \pi_* \widetilde{E} \otimes \omega_{X_0}
\end{tikzcd}
\]
at the node and using the defining property
\[
\mathcal{F} = \bigl\{ f \in \pi_* \widetilde{E} \ \big|\ A(f(x^+)) = f(x^-) \bigr\}
\]
together with the fact that \(\pi_* \widetilde{\nabla}\) is compatible with \(\nabla\). The proof of the statement $(4)$ is straightforward. This completes the proof.
\end{proof}

\subsection{Holomorphic logarithmic connections on Torsion-Free Sheaves}

We now extend the Weil-type criterion to the broader setting of torsion-free (not necessarily locally free) sheaves on the nodal curve \(X_0\).

\subsection{The Atiyah Exact Sequence}

The \emph{first jet sheaf} \(\mathcal{J}^1(\mathcal{F})\) fits into the following short exact sequence of coherent sheaves, called the \textbf{Atiyah exact sequence}:
\begin{equation}
0 \longrightarrow \mathcal{F} \otimes_{\mathcal{O}_{X_0}} \omega_{X_0}^1 \longrightarrow \mathcal{J}^1(\mathcal{F}) \longrightarrow \mathcal{F} \longrightarrow 0.
\end{equation}

This sequence can be described explicitly as follows. As a sheaf of \(\mathbb{C}\)-vector spaces (i.e., forgetting the \(\mathcal{O}_{X_0}\)-module structure), we have
\[
\mathcal{J}^1(\mathcal{F}) \ \cong \ (\mathcal{F} \otimes_{\mathcal{O}_{X_0}} \omega_{X_0}^1) \oplus \mathcal{F}.
\]
The \(\mathcal{O}_{X_0}\)-module structure is given by the rule
\[
f \cdot (s \otimes \omega,\ t) \ = \ \bigl( fs \otimes \omega + t \otimes df,\ ft \bigr)
\]
for local sections \(f \in \mathcal{O}_{X_0}\), \(s \in \mathcal{F}\), \(\omega \in \omega_{X_0}^1\), and \(t \in \mathcal{F}\), where \(df\) denotes the image of \(f\) under the universal derivation \(d \colon \mathcal{O}_{X_0} \to \omega_{X_0}^1\).

The maps in the exact sequence are:
\begin{itemize}
    \item The inclusion \(\mathcal{F} \otimes \omega_{X_0}^1 \hookrightarrow \mathcal{J}^1(\mathcal{F})\) sends \(s \otimes \omega \mapsto (s \otimes \omega,\ 0)\),
    \item The projection \(\mathcal{J}^1(\mathcal{F}) \twoheadrightarrow \mathcal{F}\) sends \((s \otimes \omega,\ t) \mapsto t\).
\end{itemize}
From the definition, it is obvious that the sequence is exact.

\begin{definition}[Atiyah class]
    The \emph{Atiyah class} of a torsion-free coherent sheaf \(\mathcal{F}\) on \(X_0\), denoted \(\operatorname{At}(\mathcal{F})\), is the extension class 
    \[
    \operatorname{At}(\mathcal{F}) := [\mathcal{J}^1(\mathcal{F})] \ \in \ \operatorname{Ext}^1_{\mathcal{O}_{X_0}}\bigl(\mathcal{F},\ \mathcal{F} \otimes_{\mathcal{O}_{X_0}} \omega_{X_0}\bigr)
    \]
    of the Atiyah exact sequence
    \[
    0 \longrightarrow \mathcal{F} \otimes \omega_{X_0} \longrightarrow \mathcal{J}^1(\mathcal{F}) \longrightarrow \mathcal{F} \longrightarrow 0.
    \]
    In particular, \(\operatorname{At}(\mathcal{F}) = 0\) if and only if this sequence splits (equivalently, if and only if \(\mathcal{F}\) admits a holomorphic logarithmic connection).
\end{definition}

\begin{definition}[Numerical first Chern class of a torsion-free sheaf]
Let \(X_0\) be a projective Gorenstein curve (possibly nodal) and let \(\mathcal{F}\) be a torsion-free coherent sheaf on \(X_0\) of rank \(r = \operatorname{rk}(\mathcal{F})\), where the rank is the dimension of the stalk of \(\mathcal{F}\) at the generic point of \(X_0\).

The \textbf{Numerical first Chern class or the degree} of \(\mathcal{F}\) is defined by the Riemann--Roch formula:
\[
c_1(\mathcal{F}) \ := \ \chi(\mathcal{F}) - r \cdot \chi(\mathcal{O}_{X_0}) \ \in \ \mathbb{Z}.
\]
\end{definition}

\begin{lemma}[Trace of the Atiyah class]
Let \(X_0\) be a Gorenstein curve and \(\mathcal{F}\) a torsion-free coherent sheaf on \(X_0\). Then
\[
\operatorname{tr} \bigl( \operatorname{At}(\mathcal{F}) \bigr) \ = \ c_1(\mathcal{F}) \ \in \ H^1(X_0, \omega_{X_0}),
\]
where \(\operatorname{tr} : \operatorname{Ext}^1(\mathcal{F}, \mathcal{F} \otimes \omega_{X_0}) \to H^1(X_0, \omega_{X_0})\) is the trace map induced by the natural trace morphism \(\operatorname{tr} : \End(\mathcal{F}) \to \mathcal{O}_{X_0}\).
\end{lemma}

\begin{proof}
We follow a deformation-theoretic argument. Let \(R\) be a discrete valuation ring with fraction field \(K\) and residue field $\mathbb C$. Then from \cite[Theorem~B.2 and Corollary~B.3, Appendix~B]{M1}, it follows that there exists a flat projective morphism \(\pi \colon \mathcal{X} \to \Spec R\) with smooth generic fiber \(\mathcal{X}_K\) and special fiber \(\mathcal{X}_k \cong X_0\).

Since \(\dim X_0 = 1\), the obstruction to deforming the coherent sheaf \(\mathcal{F}\) lies in \(Ext^2_{X_0}(\mathcal{F},\mathcal{F})\), which vanishes. Hence \(\mathcal{F}\) extends to a coherent sheaf \(\mathcal{G}\) on \(\mathcal{X}\), flat over \(R\), such that \(\mathcal{G}_K\) is a vector bundle on the smooth curve \(\mathcal{X}_K\) and \(\mathcal{G}_k \cong \mathcal{F}\).

Let \(\omega_{\mathcal{X}/R}\) be the relative dualizing sheaf. The \emph{relative Atiyah class}
\[
At(\mathcal{G}/R) \in H^1\bigl(\mathcal{X}, \End(\mathcal{G}) \otimes \omega_{\mathcal{X}/R}\bigr)
\]
is well-defined. Its trace is a global section
\[
\sigma := \tr\bigl(At(\mathcal{G}/R)\bigr) \in H^0\bigl(\Spec R,\, R^1\pi_*\omega_{\mathcal{X}/R}\bigr).
\]
By relative Serre duality, the sheaf \(R^1\pi_*\omega_{\mathcal{X}/R}\) is free of rank one over \(R\).

On the generic fiber, Atiyah's theorem gives
\[
\tr\bigl(At(\mathcal{G}_K)\bigr) = c_1(\mathcal{G}_K) \in \mathbb{Z}.
\]
Thus \(\sigma\) is the constant section with this integer value. Because flatness of \(\mathcal{G}\) over \(R\) preserves the Euler characteristic (and hence the numerical first Chern class via Riemann--Roch), the specialization of \(\sigma\) to the special fiber yields
\[
\tr\bigl(At(\mathcal{F})\bigr) = c_1(\mathcal{G}_k) = c_1(\mathcal{F})
\]
in \(H^1(X_0, \omega_{X_0})\), as required.
\end{proof}

\begin{remark}[Alternative argument using resolution of the surface]
Let \(\hat{\pi} \colon \hat{\mathcal{X}} \to \mathcal{X}\) be the minimal resolution of singularities of the surface \(\mathcal{X}\). Define
\[
\mathcal{E} := \hat{\pi}^*\mathcal{G} \Big/ tors(\hat{\pi}^*\mathcal{G}).
\]
Then \(\mathcal{E}\) is a vector bundle on the smooth surface \(\hat{\mathcal{X}}\).

The Atiyah class satisfies the compatibility
\[
\hat{\pi}_* \operatorname{At}(\mathcal{E}) = \operatorname{At}(\mathcal{G}),
\]
where the right-hand side is defined using the dualizing sheaf \(\omega_{\mathcal{X}/R}\). Equivalently, after pullback and quotienting by torsion, the Atiyah extension of \(\mathcal{G}\) on \(\mathcal{X}\) corresponds to that of \(\mathcal{E}\) on \(\hat{\mathcal{X}}\).

On the smooth surface \(\hat{\mathcal{X}}\), \(\mathcal{E}\) is locally free, so we can use the classical Čech description. If \(\{\phi_{ij}\}\) are transition matrices of \(\mathcal{E}\), then
\[
\operatorname{tr}\bigl(\operatorname{At}(\mathcal{E})\bigr) = \bigl\{ d\log(\det \phi_{ij}) \bigr\},
\]
which is the standard Čech cocycle representative of \(c_1(\mathcal{E})\) in \(H^1(\hat{\mathcal{X}}, \omega_{\hat{\mathcal{X}}})\). Hence
\[
\operatorname{tr}\bigl(\operatorname{At}(\mathcal{E})\bigr) = c_1(\mathcal{E}).
\]

Since \(\hat{\pi}\) is a resolution and \(\mathcal{G}\) is flat over \(R\), the numerical first Chern class is preserved:
\[
c_1(\mathcal{E}) = c_1(\mathcal{G})
\]
(numerically, and thus also after tracing). Restricting to the special fiber and using the compatibility of the Atiyah classes, we conclude
\[
\operatorname{tr}\bigl(\operatorname{At}(\mathcal{F})\bigr) = c_1(\mathcal{F}) \quad \text{in } H^1(X_0, \omega_{X_0}).
\]
\end{remark}

\begin{lemma}[Correspondence with Connections]
Let \(X_0\) be a Gorenstein curve and \(\mathcal{F}\) be a coherent torsion-free sheaf on \(X_0\). Then there is a natural bijection between:
\begin{itemize}
    \item The set of holomorphic logarithmic connections on \(\mathcal{F}\) with values in the dualizing sheaf, i.e., \(\mathbb{C}\)-linear morphisms
          \[
          \nabla \colon \mathcal{F} \to \mathcal{F} \otimes \omega_{X_0}
          \]
          satisfying the Leibniz rule
          \[
          \nabla(fs) = f \nabla(s) + s \otimes df
          \]
          for all local sections \(f \in \mathcal{O}_{X_0}\), \(s \in \mathcal{F}\), where \(df\) is given by the universal derivation \(\mathcal{O}_{X_0} \to \omega_{X_0}\);
    \item The set of splittings of the dualizing Atiyah sequence
          \[
          0 \longrightarrow \mathcal{F} \otimes \omega_{X_0} \longrightarrow \mathcal{J}^1(\mathcal{F}, \omega_{X_0}) \longrightarrow \mathcal{F} \longrightarrow 0.
          \]
\end{itemize}
\end{lemma}

\begin{proof}
Straightforward and we leave it for the reader.
\end{proof}

\begin{theorem}[Weil-type theorem for torsion-free sheaves]\label{prop:weil-torsionfree}
    Let \(\mathcal{F}\) be an indecomposable torsion-free coherent sheaf of rank \(n\) and degree \(0\) on the irreducible nodal curve \(X_0\). Then \(\mathcal{F}\) admits a holomorphic logarithmic connection
    \[
    \nabla \colon \mathcal{F} \to \mathcal{F} \otimes \omega_{X_0}
    \]
    if and only if \(\deg \mathcal{F} = 0\).
\end{theorem}

\begin{proof}
The obstruction to the existence of a holomorphic logarithmic connection on the torsion-free sheaf \(\mathcal{F}\) is the Atiyah class
\[
\operatorname{At}(\mathcal{F}) \in \operatorname{Ext}^1(\mathcal{F}, \mathcal{F} \otimes \omega_{X_0}).
\]
It suffices to show that \(\operatorname{At}(\mathcal{F}) = 0\) if and only if \(\deg \mathcal{F} = 0\).

The trace morphism \(\operatorname{tr} \colon \End(\mathcal{F}) \to \mathcal{O}_{X_0}\) induces a natural pairing
\[
H^0(X_0, \End(\mathcal{F})) \times \operatorname{Ext}^1(\mathcal{F}, \mathcal{F} \otimes \omega_{X_0}) \xrightarrow{\cup} \operatorname{Ext}^1(\mathcal{F}, \mathcal{F} \otimes \omega_{X_0}) \xrightarrow{Tr}  H^1(X_0, \omega_{X_0}) \cong \mathbb{C}.
\]

Since \(\mathcal{F}\) is indecomposable and torsion-free, the Fitting lemma \cite[Proposition~7.4, p.~441]{Lang} implies that every global endomorphism \(\phi \in H^0(X_0, \End(\mathcal{F}))\) can be written as \(\phi = \lambda \cdot \id + N\), where \(\lambda \in \mathbb{C}\) and \(N\) is nilpotent (note that \(\End(\mathcal{F})\) is finite-dimensional).

We first consider the nilpotent part \(N\). Because \(N\) is nilpotent, there exists a complete flag of saturated torsion-free subsheaves
\[
0 = \mathcal{F}_0 \subset \mathcal{F}_1 \subset \cdots \subset \mathcal{F}_r = \mathcal{F},
\]
where each \(\mathcal{F}_i\) is a saturated torsion-free subsheaf of rank \(i\) (i.e., \(\mathcal{F}/\mathcal{F}_i\) is torsion-free) and \(N(\mathcal{F}_i) \subset \mathcal{F}_{i-1}\) for all \(i\) (see Remark \ref{flag2}). 

First of all, from the remark \ref{Nilp}, it follows that
\[
\langle N, \operatorname{At}(\mathcal{F}) \rangle = 0.
\]

On the other hand, for the scalar part we have
\[
\langle \lambda \cdot \id, \operatorname{At}(\mathcal{F}) \rangle = \lambda \cdot \operatorname{tr}(\operatorname{At}(\mathcal{F})) = \lambda \cdot c_1(\mathcal{F}),
\]
where \(c_1(\mathcal{F})\) denotes the first Chern class of \(\mathcal{F}\).

Hence the pairing \(\langle \phi, \operatorname{At}(\mathcal{F}) \rangle = 0\) for every \(\phi \in H^0(X_0, \End(\mathcal{F}))\) if and only if \(c_1(\mathcal{F}) = 0\), which is equivalent to \(\deg \mathcal{F} = 0\). This completes the proof.
\end{proof}

\begin{remark}\label{flag2}
Such a flag can be constructed as follows. At the generic point \(\eta\) of \(X_0\), the stalk \(\mathcal{F}_\eta\) is a vector space of dimension \(r\) and \(N_\eta\) is nilpotent, so there exists a complete flag
\[
0 = V_0 \subset V_1 \subset \cdots \subset V_r = \mathcal{F}_\eta
\]
with \(\dim V_i = i\) and \(N(V_i) \subset V_{i-1}\).

We extend this flag inductively to saturated subsheaves on \(X_0\). Suppose \(\mathcal{F}_{i-1} \subset \mathcal{F}\) is already constructed, saturated of rank \(i-1\), with generic fiber \(V_{i-1}\), and satisfying \(N(\mathcal{F}_{i-1}) \subset \mathcal{F}_{i-2}\). Let \(\mathcal{Q} = \mathcal{F}/\mathcal{F}_{i-1}\), a torsion-free sheaf of rank \(r-i+1\). The image \(W = V_i/V_{i-1} \subset \mathcal{Q}_\eta\) is one-dimensional and satisfies \(N(W) = 0\) in \(\mathcal{Q}_\eta\).

Since \(\mathcal{Q}\) is torsion-free on a curve, one can always find a rank-$1$ saturated subsheaf \(\mathcal{L} \subset \mathcal{Q}\) whose generic fiber is exactly \(W\) (take a generic section spanning \(W\) and saturate). Let \(\mathcal{F}_i\) be the preimage of \(\mathcal{L}\) in \(\mathcal{F}\). Then \(\mathcal{F}_i\) is saturated of rank \(i\), \(\mathcal{F}/\mathcal{F}_i\) is torsion-free, and \(N(\mathcal{F}_i) \subset \mathcal{F}_{i-1}\).
\end{remark}

\begin{remark}\label{Nilp}
Since \( N \) is a global nilpotent endomorphism, the class \( \operatorname{Trace}(N \cup \operatorname{At}(F)) \) lies in \( H^1(X_0, \omega_{X_0}) \). Therefore it can be expressed as a cohomology class \([\{\sigma_{ij}\}]\) of a Čech 1-cocycle with values in \( \omega_{X_0} \). We want to compute the actual local sections \( \sigma_{ij} \) (not just the cohomology class), for a particularly chosen representative, and we want to do this outside the node.

Let \( U = X_0 \setminus \{p\} \) be the smooth locus, and let \( F|_U \) be the vector bundle obtained by restricting \( F \). On \( U \), since \( F \) is locally free, \( \operatorname{At}(F|_U) \) is the usual Atiyah class of a vector bundle. Choose an open cover \( \{U_i\} \) of \( U \) with local frames \( e_i \) for \( F|_{U_i} \). Let
\[
g_{ij} \in \Gamma(U_{ij}, \mathrm{GL}(r,\mathcal{O}))
\]
be the transition matrices:
\[
e_i = g_{ij} \, e_j \quad \text{on } U_{ij}.
\]
Then a Čech 1-cocycle representative for \( \operatorname{At}(F|_U) \) is
\[
\alpha_{ij} = g_{ij}^{-1} \, dg_{ij} \quad \in \ \Gamma(U_{ij}, \operatorname{End}(F) \otimes \Omega^1_U).
\]

Since \( N \) is nilpotent, it admits a complete flag
\[
0 = F_0 \subset F_1 \subset \cdots \subset F_r = F
\]
such that \( N(F_k) \subset F_k \) for all \( k \).

Because the frames on \( U_i \) and \( U_j \) both respect the same global flag, the transition matrix \( g_{ij} \) (defined by \( e_i = g_{ij} e_j \)) must preserve the flag. Hence \( g_{ij} \) is upper triangular.

The usual Atiyah cocycle is
\[
\alpha_{ij} = g_{ij}^{-1} \, dg_{ij}.
\]
Since \( g_{ij} \) is upper triangular, \( g_{ij}^{-1} \) is also upper triangular, and therefore \( \alpha_{ij} = g_{ij}^{-1} dg_{ij} \) is upper triangular.

Now multiply on the left by \( N_i \) (strictly upper triangular):
\[
N_i \cdot \alpha_{ij} \quad = \quad \text{strictly upper triangular matrix},
\]
because the product of a strictly upper triangular matrix and an upper triangular matrix is strictly upper triangular.

Therefore the cocycle
\[
(N \cup \operatorname{At}(F))_{ij} := N_i \cdot (g_{ij}^{-1} dg_{ij})
\]
takes values in strictly upper triangular matrices on every overlap \( U_{ij} \).

This representative shows that, with respect to a flag-adapted cover on the smooth locus, the cocycle for \( N \cup \operatorname{At}(F) \) (and hence its trace) is visibly strictly upper triangular, and in particular its trace \( \sigma_{ij} \) vanishes identically on \( U_{ij} \). This shows that the class $Trace(N\cup At(F))=0$. 
\end{remark}

\begin{theorem}
    Let $\mathcal F$ be an indecomposable torsion-free sheaf of rank $n$ and degree $0$. From Proposition \ref{Trip}, it follows that there exists a unique triple $(\tilde{E}, A)$ on the normalisation $\widetilde{X_0}$ which induces the torsionfree sheaf $\mathcal F$ on the nodal curve $X_0$. 
    
    The torsionfree sheaf $\mathcal F$ admits a holomorphic logarithmic connection $\nabla: \mathcal F\to \mathcal F\otimes \omega_{X_0}$ such that the induced logarithmic connection triple $(\widetilde{E}, A, \widetilde{\nabla}: \widetilde{E}\to \widetilde{E}\otimes \Omega^1_{\widetilde{X_0}}(x^++x^-))$ has residues of the form $\lambda \cdot I_{\widetilde{E}_{x^+}}$ at $x^+$ and $-\lambda \cdot I_{\widetilde{E}_{x^-}}$ at $x^-$, for some scalar $\lambda \in \mathbb{C}$.
\end{theorem}

\begin{proof}
By Proposition~\ref{Trip}, there exists a unique (up to unique isomorphism) triple \((\widetilde{E}, A, \widetilde{\nabla})\) consisting of a vector bundle \(\widetilde{E}\) on the normalization \(\widetilde{X_0}\), a linear map \(A \colon \widetilde{E}_{x^+} \to \widetilde{E}_{x^-}\), and a logarithmic connection \(\widetilde{\nabla}\) on \(\widetilde{E}\) inducing the given holomorphic logarithmic connection on \(\mathcal{F}\). It remains to show that we may choose the connection so that
\[
\operatorname{res}_{x^+} \widetilde{\nabla} = \lambda \cdot I_{\widetilde{E}_{x^+}}, \qquad
\operatorname{res}_{x^-} \widetilde{\nabla} = -\lambda \cdot I_{\widetilde{E}_{x^-}}
\]
for some \(\lambda \in \mathbb{C}\).

Let \(\mathcal{G} \subset \pi_*(\End(\widetilde{E}))\) be the subsheaf of endomorphisms compatible with the opposite residue condition. More precisely, \(\mathcal{G}\) is the kernel of the natural surjection
\[
\pi_*(\End(\widetilde{E})) \twoheadrightarrow \frac{\pi_*(\End(\widetilde{E}_{x^+}) \oplus \End(\widetilde{E}_{x^-}))}{\operatorname{im}(\phi)},
\]
where \(\phi \colon \mathbb{C} \to \End(\widetilde{E}_{x^+}) \oplus \End(\widetilde{E}_{x^-})\) is the map
\[
\lambda \cdot I \;\longmapsto\; (\lambda \cdot I_{\widetilde{E}_{x^+}},\; -\lambda \cdot I_{\widetilde{E}_{x^-}}).
\]
Equivalently, \(\mathcal{G}\) sits in the short exact sequence (see proof of theorem \ref{Weil2})
\[
0 \to \mathcal Hom(\mathcal F, \mathcal{F} \otimes I_x) \to \mathcal{G} \to \mathbb{C}\cdot(\id_{x^+},-\id_{x^-}) \to 0.
\]

The existence of a holomorphic logarithmic connection on \(\mathcal{F}\) with the desired scalar residue property is equivalent to the vanishing of the obstruction class
\[
\beta_{\mathcal{F}} \in H^1(X_0, \omega_{X_0} \otimes \mathcal{G}).
\]

By Serre duality there is a non-degenerate pairing
\[
\langle -, - \rangle \colon H^0(X_0, \mathcal{G}^\vee) \times H^1(X_0, \mathcal{G}\otimes \omega_{X_0}) \to H^1(X_0, \omega_{X_0}) \cong \mathbb{C}.
\]

We clarify the computation of \(H^0(X_0, \mathcal{G}^\vee)\).

Let \(\mathcal{G}\) sit in the short exact sequence
\[
0 \to K \to \mathcal{G} \to Q \to 0,
\]
where \(K = \mathcal Hom(\mathcal F, \mathcal{F} \otimes I_x)\) and \(Q = \mathbb{C} \cdot (\id_{x^+}, -\id_{x^-})\) (skyscraper sheaf at the node).

Dualizing gives
\[
0 \to \mathcal{G}^\vee \to K^\vee \xrightarrow{\alpha} \mathcal Ext^1(Q, \mathcal{O}_{X_0}) \simeq \mathbb{C} \to \cdots.
\]
Hence on global sections,
\[
H^0(X_0, \mathcal{G}^\vee) = \ker(\alpha),\qquad \text{where }\alpha\colon H^0(K^\vee) \to \mathbb{C}.
\]

By tensor-hom adjunction,
\[
K^\vee = (\mathcal Hom(\mathcal F, \mathcal{F} \otimes I_x))^\vee \ \cong\ (\mathcal End (\mathcal F)\otimes I_x)^\vee \ \cong\ \mathcal Hom(\End(\mathcal{F}), I_x^\vee). \footnote{The natural map 
\[
\End(\mathcal{F}) \otimes_{\mathcal{O}_{X_0}} I_x 
\longrightarrow 
\mathcal{H}om(\mathcal{F}, \mathcal{F} \otimes_{\mathcal{O}_{X_0}} I_x)
\]
is in general neither injective nor surjective. More precisely, there is a four-term exact sequence
\[
0 \to T_1 \to \End(\mathcal{F}) \otimes_{\mathcal{O}_{X_0}} I_x 
\longrightarrow 
\mathcal{H}om(\mathcal{F}, \mathcal{F} \otimes_{\mathcal{O}_{X_0}} I_x) \to T_2 \to 0,
\]
where \(T_1\) and \(T_2\) are skyscraper sheaves (torsion sheaves) supported at the node. In particular, upon taking the dual sheaf we obtain a canonical isomorphism
\[
\bigl( \End(\mathcal{F}) \otimes_{\mathcal{O}_{X_0}} I_x \bigr)^\vee 
\;\simeq\; 
\mathcal{H}om(\mathcal{F}, \mathcal{F} \otimes_{\mathcal{O}_{X_0}} I_x)^\vee,
\]
since dualizing kills skyscraper sheaves.}
\]

Using the canonical isomorphism \(I_x^\vee \cong \pi_*\mathcal{O}_{\widetilde{X_0}}\), we have
\[
K^\vee \ \cong\ \mathcal Hom(\End(\mathcal{F}), \pi_*\mathcal{O}_{\widetilde{X_0}}).
\]

Equivalently, since the endomorphisms of \(\mathcal{F}\) are precisely those endomorphisms of the corresponding vector bundle in normalization that preserve the gluing submodule \(\Gamma_A\) (i.e., the fiber identification graph $A: \tilde{E}_{x^+}\to \tilde{E}_{x^-}$), we can also write
\[
H^0(X_0, \mathcal{G}^\vee) \ = \ \ker \Bigl(\Hom_{_{\Gamma_A}}(\End(\widetilde{E}), \mathcal{O}_{\widetilde{X_0}}) \to \mathbb{C} \Bigr),
\]
where \(\Hom_{\Gamma_A}\) denotes morphisms that preserve the gluing condition defining \(\mathcal{F}\). In particular, $\phi$ induces a endomorphism of $\mathcal F$ as well. 

Concretely, this kernel consists of those global endomorphisms \(\phi\) of the vector bundle on the normalization (corresponding to \(\mathcal{F}\)) such that
\[
\operatorname{tr}(\phi_{x^+}) = -\operatorname{tr}(\phi_{x^-})
\]
and \(\phi\) preserves the gluing data of \(\mathcal{F}\).

Since \(\mathcal{F}\) is indecomposable and torsion-free, every global endomorphism \(\phi \in H^0(X_0, \End(\mathcal{F}))\) (which lifts to an endomorphism of \(\widetilde{E}\)) is of the form \(\phi = \lambda \cdot \id + N\) with \(\lambda \in \mathbb{C}\) and \(N\) nilpotent (Fitting lemma). If moreover \(\phi \in H^0(X_0, \mathcal{G}^\vee)\), the trace condition forces \(\lambda = 0\), so \(\phi = N\) is nilpotent.

It remains to show \(\langle N, \beta_{\mathcal{F}} \rangle = 0\) for any such nilpotent \(N\). This follows from the same strategy discussed in the Remark \ref{Nilp}. 

Since the pairing vanishes on all of \(H^0(X_0, \mathcal{G}^\vee)\), non-degeneracy of Serre duality implies \(\beta_{\mathcal{F}} = 0\). This completes the proof.

\end{proof}

\section{Example: Connections on some indecomposable bundles on the irreducible rational nodal curve}
\label{sec:nodal-elliptic}

Let \(X_0 \subset \mathbb{P}^2\) be the irreducible nodal cubic defined by \(y^2 z = x^3 + x^2 z\). It has a single node at \([0:0:1]\) and an arithmetic genus \(g=1\). Its normalization is \(\pi \colon \mathbb{P}^1 \to X_0\), and the dualizing sheaf satisfies \(\omega_{X_0} \cong \mathcal{O}_{X_0}\).

According to results of Bodnarchuk--Burban--Drozd--Greuel \cite{BBDG2006}, \begin{itemize}
\item Every indecomposable vector bundle \(E\) of rank \(n\) and degree \(0\) on the nodal curve \(X_0\) is isomorphic to
  \[
  E \cong (\pi_n)_* L \otimes F_n,
  \]
  where \(F_n\) is the Atiyah bundle of rank \(n\) on \(X_0\), and \(L\) is a line bundle of degree \(0\) on a suitable finite étale cover \(\pi_n : Y_n \to X_0\).

\item The pullback of \(F_n\) to the normalization \(\tilde{X}_0 \cong \mathbb{P}^1\) is trivial:
  \[
  \nu^* F_n \cong \mathcal{O}_{\mathbb{P}^1}^{\oplus n}.
  \]

\item The bundle \(F_n\) is recovered by gluing the fibers over the two preimages \(x^+, x^-\) of the node via the unipotent Jordan block
  \[
  A = J_n(1) =
  \begin{pmatrix}
  1 & 1 & 0 & \cdots & 0 \\
  0 & 1 & 1 & \cdots & 0 \\
  \vdots & \vdots & \ddots & \ddots & \vdots \\
  0 & 0 & \cdots & 1 & 1 \\
  0 & 0 & \cdots & 0 & 1
  \end{pmatrix}.
  \]
\end{itemize}

Since \(\deg F_n = 0\), Theorem~\ref{Weil} guarantees the existence of holomorphic logarithmic connections on \(F_n\). An explicit one-parameter family can be constructed as follows.

Equip the trivial bundle \(\mathcal{O}_{\mathbb{P}^1}^{\oplus n}\) with the logarithmic connection
\[
\widetilde{\nabla}_t =
\begin{cases}
d + t \frac{dz}{z} \cdot I_n & \text{near } x^+, \\
d - t \frac{dz}{z} \cdot I_n & \text{near } x^-,
\end{cases}
\]
for any \(t \in \mathbb{C}\). The residues are \(\operatorname{res}_{x^+} \widetilde{\nabla}_t = t I_n\) and \(\operatorname{res}_{x^-} \widetilde{\nabla}_t = -t I_n\).

The compatibility condition
\[
-A \circ \operatorname{res}_{x^-}(\widetilde{\nabla}_t) = \operatorname{res}_{x^+}(\widetilde{\nabla}_t) \circ A
\]
holds because residue matrices are scalar matrices. Thus \(\widetilde{\nabla}_t\) descends to a holomorphic logarithmic connection \(\nabla_t\) on \(F_n\) over \(X_0\). All such connections are flat (as \(\dim X_0 = 1\)).

When \(t=0\), \(\nabla_0\) is the pullback of the trivial connection on \(\mathbb{P}^1\). For \(t \neq 0\), \(\nabla_t\) has a non-trivial logarithmic singularity at the node with scalar residue \(t\).

The space of holomorphic logarithmic connections on \(F_n\) is an affine space over
\[
H^0(X_0, \End(F_n) \otimes \omega_{X_0}).
\]
Since \(\deg(\End(F_n) \otimes \omega_{X_0}) = 0\) and \(\operatorname{rk}(\End(F_n)) = n^2\), Riemann--Roch gives
\[
\chi(\End(F_n) \otimes \omega_{X_0}) = 0.
\]
By Serre duality and the fact that \(H^0(X_0, \End(F_n)) \cong \mathbb{C}\) (see \cite[Lemma 9.13]{BBDG2006}), we have
\[
\dim H^1(X_0, \End(F_n) \otimes \omega_{X_0}) = 1,
\]
hence \(\dim H^0(X_0, \End(F_n) \otimes \omega_{X_0}) = 1\). This confirms that the family \(\{\nabla_t \mid t \in \mathbb{C}\}\) exhausts all holomorphic logarithmic connections (up to translation by flat ones, which are unique up to scalar in this case).

\begin{remark}
On a smooth elliptic curve, the Atiyah bundle of rank \(n\) admits a one-dimensional space of flat connections, parametrized by \(H^0(X, \Omega_X) \cong \mathbb{C}\). 

In the nodal degeneration to the rational nodal curve \(X_0\), this one-dimensional parameter survives and is realized concretely as the residue \(t\) at the node (i.e., the off-diagonal entry in the Jordan block gluing).

Consequently, an indecomposable vector bundle \(E\) of rank \(n\) and degree \(0\) on the rational nodal curve \(X_0\) admits a one-dimensional space of holomorphic logarithmic connections with at most logarithmic poles at the node if and only if \(\dim H^0(X_0, \End(E)) = 1\), which holds precisely when \(E\) is isomorphic to the Atiyah bundle \(F_n\) \cite[Lemma 8 and Theorem 21]{BBDG2006}.
\end{remark}

\section*{Funding Declarations}
This research received no external funding.

\end{document}